\documentclass{article}
\usepackage{amsmath, amsthm, amssymb, latexsym,epsfig,amsthm,multicol,wasysym,enumitem,bbm}
\usepackage{graphicx} 
\usepackage{xcolor}
\usepackage{hyperref}
\title{The Sobolev extension problem on trees and in the plane}
\author{Jacob Carruth and Arie Israel}
\date{June 2024}

\graphicspath{ {./images/} }

\newcommand{\cC}{\mathcal{C}}
\newcommand{\cM}{\mathcal{M}}
\newcommand{\cW}{\mathcal{W}}
\newcommand{\cQ}{\mathcal{Q}}

\newcommand{\cS}{\mathcal{S}}
\newcommand{\cH}{\mathcal{H}}

\newcommand{\R}{\mathbb{R}}
\newcommand{\Z}{\mathbb{Z}}

\newcommand{\diam}{\mathrm{diam}}

\newcommand{\touch}{\leftrightarrow}
\newcommand{\dist}{\mathrm{dist}}
\newcommand{\supp}{\mathrm{supp}}
\newcommand{\lca}{\mathrm{lca}}

\newcommand\extrafootertext[1]{%
    \bgroup
    \renewcommand\thefootnote{\fnsymbol{footnote}}%
    \renewcommand\thempfootnote{\fnsymbol{mpfootnote}}%
    \footnotetext[0]{#1}%
    \egroup
}

\newtheorem{thm}{Theorem}
\newtheorem{lem}{Lemma}
\newtheorem{defn}{Definition}

\begin{document}

\maketitle

\begin{abstract}
    Let $V$ be a finite tree with radially decaying weights. We show that there exists a set $E \subset \R^2$ for which the following two problems are equivalent: (1) Given a (real-valued) function $\phi$ on the leaves of $V$, extend it to a function $\Phi$ on all of $V$ so that $||\Phi||_{L^{1,p}(V)}$ has optimal order of magnitude. Here, $L^{1,p}(V)$ is a weighted Sobolev space on $V$. (2) Given a function $f:E \rightarrow \R$, extend it to a function $F \in L^{2,p}(\R^2)$ so that $||F||_{L^{2,p}(\R^2)}$ has optimal order of magnitude.
\end{abstract}

\section{Introduction}\label{sec: intro}

Let $L^{m,p}(\R^n)$ denote the homogeneous Sobolev space of real-valued functions on $\R^n$ whose (distributional) derivatives of order $m$ belong to $L^p(\R^n)$ for $1< p< \infty$. This space is equipped with the seminorm
\[
||F||_{L^{m,p}(\R^n)} = \bigg( \int_{\R^n} | \nabla^m F(x)|^p \ dx \bigg)^{1/p}.
\]
Provided $p > n/m$,  any $F \in L^{m,p}(\R^n)$ is a continuous function, and therefore can be restricted to an arbitrary subset $\Omega \subset \R^n$. We thus define the trace seminorm for functions $f: \Omega \rightarrow \R$ by
\[
||f||_{L^{m,p}(\Omega)} = \inf\{ \| F \|_{L^{m,p}(\R^n)} : F \in L^{m,p}(\R^n), \; F|_\Omega =f\},
\]
and we define the trace space $L^{m,p}(\Omega)$ to be the set of all functions $f: \Omega \rightarrow \R$ with finite trace norm. We say that an operator $T: L^{m,p}(\Omega) \rightarrow L^{m,p}(\R^n)$ is an \emph{extension operator} if $Tf|_\Omega = f$ for every $f \in L^{m,p}(\Omega)$.\extrafootertext{This work was supported by AFOSR grant FA9550-19-1-0005.}

\sloppypar In this article, we consider the \emph{Sobolev extension problem}: Given a finite subset $\Omega \subset \R^n$, does there exist a bounded linear extension operator ${T: L^{m,p}(\Omega) \rightarrow L^{m,p}(\R^n)}$ satisfying $||Tf||_{L^{m,p}(\R^n)} \le C||f||_{L^{m,p}(\Omega)}$ for some constant $C=C(m,n,p)$ (in particular, $C$ is independent of $\Omega$)?

When $p> n$ and $m$ is arbitrary, the second-named author, C. Fefferman, and G.K. Luli \cite{fefferman2014sobolev} completely resolved this problem in the affirmative.

When $n/m < p \le n$, however, little is known. In this article, we consider the first nontrivial case in this parameter range -- we study the Sobolev extension problem for the space $L^{2,p}(\R^2)$ when $1 < p < 2$. (Note that the problem is well-understood when $p=2$, because $L^{2,2}(\R^2)$ is a Hilbert space.) We refer to this as the \emph{planar Sobolev extension problem}.

For the remainder of this article, we assume that $1< p< 2$. We now survey what is known about the planar Sobolev extension problem. (Our focus here is on the case in which the set $\Omega$ is \emph{finite}; for interesting results when $\Omega$ is a bounded, simply connected domain, see \cite{shvartsman2016planar}.)

Recently, M. Drake, C. Fefferman, K. Ren, and A. Skorobogatova \cite{drake2024sobolev} showed that there is a bounded linear extension operator $T:L^{2,p}(\Omega) \rightarrow L^{2,p}(\R^2)$ when $\Omega$ is a finite subset of a line in $\R^2$. Moreover, the norm of their extension operator depends only on $p$, as desired.

In our previous paper \cite{carruth2023example}, we constructed a bounded linear extension operator $T: L^{2,p}(\Omega) \rightarrow L^{2,p}(\R^2)$ for $\Omega$ belonging to a certain family of discrete subsets of $\R^2$ with fractal geometry. We showed that the construction of such an operator could be reduced to an extension problem for a weighted Sobolev space on a tree. Thanks to a theorem of Fefferman-Klartag \cite{fefferman2023linear}, we were able to solve the extension problem on the tree, and thus construct a linear extension operator for the Sobolev space on the plane.

In this article, we continue to investigate the connection between the planar Sobolev extension problem and weighted Sobolev extension problem on trees. The main theorem of this paper establishes conditions under which these problems are equivalent.

Consider a rooted $N$-ary tree of depth $L \ge 1$ with vertices $V$. By $N$-ary, we mean that every non-leaf node has \emph{at most} $N$ children. In addition, to avoid degenerate branches, we require each non-leaf node to have at least 2 children. We'll abuse notation and refer to $V$ as the tree. We let $d(v)$ denote the depth of $v \in V$.

We write $[N] = \{0,1,\dots, N-1\}$. We fix an ordering of the tree, i.e., an isomorphism from $V$ to a subset of
\[
\bigcup_{k=0}^L [ N]^k
\]
so that any $v \in V$ is identified with a string of $d(v)$ digits from the set $[N]$. 

The root node of $V$ is the empty string $\emptyset$ of length zero. We write $V_0 = V \backslash \{\emptyset\}$.

For $v \in V_0$ and $1 \le k \le d(v)$, let $v_k$ denote the $k$-th entry of $v$ and let $\pi_k(v) \in V_0$ denote the prefix of $v$ of length $k$. We define $\pi_0(v) = \emptyset$ and write $\pi(v) = \pi_{d(v) - 1}(v)$ to denote the parent of $v \in V_0$. We denote the set of leaves of $V$ by $\partial V$.

Given vertices $v_0, v_1$ in $V$, if $\pi_{d(v_0)}(v_1) = v_0$ then we say that $v_1$ is a \emph{descendent} of $v_0$ and that $v_0$ is an \emph{ancestor} of $v_1$. In particular, each vertex of $V$ is both an ancestor and a descendent of itself. We let $\lca(x,y)$ denote the \emph{lowest common ancestor} of $x,y \in V$, namely, the ancestor of $x$ and $y$ of largest depth.

We suppose that we are given a set of \emph{weights} $\{W_v\}_{v \in V_0}$, where $W_v > 0$ for every $v \in V_0$. We define the $L^{1,p}(V)$-seminorm of $\Phi: V \rightarrow \R$ by
\[
||\Phi||_{L^{1,p}(V)} = \Big(\sum_{ v \in V_0} |\Phi(v) - \Phi(\pi(v))|^p \cdot  W_v^{2-p}\Big)^{1/p},
\]
and the $L^{1,p}(\partial V)$ trace seminorm of $\phi: \partial V \rightarrow \R$ by
\[
||\phi||_{L^{1,p}(\partial V)} = \inf \{ || \Phi||_{L^{1,p}(V)}: \Phi|_{\partial V} = \phi\}.
\]
We write $L^{1,p}(V)$, $L^{1,p}(\partial V)$ to denote the spaces of real-valued functions on (respectively) $V$, $\partial V$, equipped with the relevant seminorm. We say that an operator $H : L^{1,p}(\partial V) \rightarrow L^{1,p}(V)$ is an \emph{extension operator} if $H\phi|_{\partial V} = \phi$ for all $\phi: \partial V \rightarrow \R$.

We now state the \emph{weighted Sobolev extension problem on trees}: For any $N$-ary tree $V$, as above, does there exist a bounded linear extension operator $H:L^{1,p}(\partial V) \rightarrow L^{1,p}(V)$ satisfying 
\[
||H\phi||_{L^{1,p}( V)} \le C ||\phi||_{L^{1,p}(\partial V)}
\]
for a constant $C = C(p,N)$ (i.e., $C$ is independent of $V$ and the weights $\{W_v\}_{v\in V_0}$)?

We say that an $N$-ary tree is \emph{perfect} if each non-leaf node has exactly $N$ children and all leaf nodes are at the same depth. We say that weights $\{W_v\}_{v \in V_0}$ are \emph{radially symmetric} if $W_v = W_u$ for every $v,u\in V_0$ with $d(v) = d(u)$.

Thanks to the work of Fefferman and Klartag \cite{fefferman2023linear} mentioned above, such an operator $H$ is known to exist when $V$ is a perfect, binary tree with radially symmetric weights. Additionally, in \cite{bjorn2017geometric}, A.\ Bj\"{o}rn, J.\ Bj\"{o}rn, J.\ Gill, and N.\ Shanmugalingam  show that $H$ can be taken to be a simple averaging operator when $V$ is a perfect tree with radially symmetric weights satisfying certain additional properties. These are the only results that we are aware of on the problem of weighted Sobolev extension on trees. We emphasize that, to our knowledge, nothing is known for finite trees when either (1) the tree $V$ is not perfect or (2) the weights are not radially symmetric.

In this article, we make neither of these assumptions. Instead, we introduce a parameter $\varepsilon \in (0,1)$ and say that weights $\{W_v\}_{v \in V_0}$ are \emph{radially decaying} provided
\begin{equation}\label{eq: weight cond}
    W_v \le \varepsilon W_{\pi(v)}\; \text{for all}\; v \in V_0.
\end{equation}
Here and in the remainder of this paper, we adopt the convention that $W_{\emptyset}=1$. Clearly, for such radially decaying weights we have
\begin{equation}
    W_{v_1} \leq \varepsilon^{d(v_1) - d(v_0)} W_{v_0} \mbox{ if } v_1 \mbox{ is a descendent of } v_0 \mbox{ in } V.
    \label{eqn: rad_decay}
\end{equation}
We then have the following theorem.

\begin{thm}\label{thm: main}
There exists an absolute constant $k_0 >0$ so that the following holds. Fix $N \geq 2$. Let $V$ be an $N$-ary tree, and let $\{W_v\}_{v\in V_0}$ be radially decaying weights satisfying \eqref{eq: weight cond} for some $\varepsilon \le k_0/ N$. Then there exists a set $E \subset \R^2$ such that the following holds: 

For any $1<p< 2$, there exists a bounded linear extension operator $H: L^{1,p}(\partial V)\rightarrow L^{1,p}(V)$ if and only if there exists a bounded linear extension operator $T: L^{2,p}(E) \rightarrow L^{2,p}(\R^2)$. 
    
In addition, if such operators exist, then
    \[
    C^{-1} ||T||_{L^{2,p}(E)\rightarrow L^{2,p}(\R^2)} \le ||H||_{L^{1,p}(\partial V)\rightarrow L^{1,p}(V)} \le C ||T||_{L^{2,p}(E)\rightarrow L^{2,p}(\R^2)}
    \]
    for a constant $C= C(p,N)$.
\end{thm}

Thanks to Theorem \ref{thm: main}, a negative answer to the problem of Sobolev extension on trees with radially decaying weights would resolve the planar Sobolev extension problem in the negative. This would be the first known example of a negative answer to the general Sobolev extension problem.

Alternatively, a positive answer to the problem of Sobolev extension on trees with radially decaying weights would produce the first known example of a bounded linear extension operator $T: L^{2,p}(E)\rightarrow L^{2,p}(\R^2)$ for certain sets $E \subset \R^2$.

We remark that in our previous paper \cite{carruth2023example}, we showed that for a certain set $E \subset \R^2$ there exists a bounded linear extension operator $L^{2,p}(E)\rightarrow L^{2,p}(\R^2)$ if there exists a bounded linear extension operator $L^{1,p}(\partial V) \rightarrow L^{1,p}(V)$ for a certain full, binary, weighted tree. (Note that we did \emph{not} show that the extension problems are equivalent.) Theorem \ref{thm: main} improves this result by (1) allowing for much more general trees and (2) establishing the equivalence of the extension problems.

For the remainder of this article we place ourselves in the setting of Theorem \ref{thm: main}: We let $k_0 > 0$ be a small enough absolute constant, to be picked later, and we fix an integer $N \ge 2$, a rooted $N$-ary tree $V$ (of which we fix some ordering), and radially decaying weights $\{W_v\}_{v \in V_0}$ satisfying \eqref{eq: weight cond} for some $0 < \varepsilon \le k_0 / N$.

We now construct the set $E \subset \R^2$ whose existence is asserted by Theorem \ref{thm: main}. Define
\begin{equation}\label{eq: delta def}
 \Delta = \min_{v \in \partial V} W_v
\end{equation}
and recursively define a map $\Psi: V \rightarrow \R$ via
\begin{equation}\label{eq: psi def}
\Psi(v) = \begin{cases}
    0 &\text{if } v = \emptyset,\\
    \Psi(\pi(v)) + W_{\pi(v)}\cdot \frac{v_{d(v)}}{N-1} &\text{else}.
\end{cases}
\end{equation}
Observe that
\begin{equation}\label{eq: psi 2}
    \Psi(v) = \sum_{i = 1}^{d(v)} W_{\pi_{i-1}(v)}\cdot\frac{v_i}{N-1}\;\text{for any}\; v \in V_0.
\end{equation}
The set $E$ is then of the form
\[
E = E_1 \cup E_2,
\]
where
\begin{align}
E_1 &= ([0,2) \cap (\Delta \Z)) \times \{0\},\label{defn: E1}\\
E_2 &= \{ (\Psi(v), W_v) : v \in \partial V\}.\label{defn: E2}
\end{align}
See Figure \ref{fig1} for an illustration of $E$ corresponding to a specific weighted tree of depth 2.

This concludes the introduction; the remainder of this article is devoted to the proof of Theorem \ref{thm: main}.

We thank Marjorie Drake, Charles Fefferman, Bo'az Klartag, Kevin Ren, Pavel Shvartsman, Anna Skorobogatova, and Ignacio Uriarte-Tuero for helpful conversations.

\begin{figure}[h]\label{fig1}
\includegraphics[width=5cm]{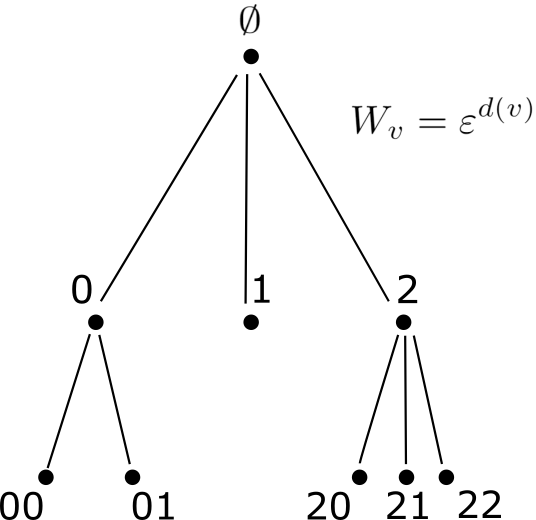}
\includegraphics[width=7cm]{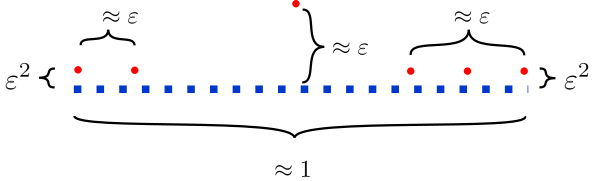}
\caption{A weighted tree $V$ of depth $2$ and the accompanying set $E = E_1 \cup E_2$. Points of $E_1$ are depicted by a sequence of blue squares of spacing $\approx \epsilon^2$, while points of $E_2$ are marked by $6$ red dots.}
\end{figure}

\section{Preliminaries}\label{sec: prelim}

Throughout this article, we will write $K, K', k, k', \dots$ to denote positive absolute constants (independent of $p$ and all other parameters), and $K_X, K'_X, \dots$ to denote positive constants depending on a parameter $X$. The value of these constants may change from line to line. For  $A, B > 0$ we write $A \lesssim B$ (resp. $A \lesssim_X B$) if there exists a constant $K$ (resp.\ $K_X$) such that $A \le K B$ (resp.\ $A \le K_X B$). We write $A \approx B$ (resp.\ $A \approx_X$ B) if $A \lesssim B$ and $A\lesssim B$ (resp.\ $A\lesssim_X B$ and $B \lesssim_X A$).

Given $\delta > 0$, we say a set $S \subset \R^2$ is \emph{$\delta$-separated} provided $|x-y| \geq \delta$ for all distinct $x,y \in S$.

For a (Lebesgue) measurable function $F$ defined on a measurable set $S \subset \R^2$ with $|S| > 0$, we write $(F)_S := |S|^{-1} \int_S F \ dx$.

Given an annulus $A = \{ x \in \R^2 : r \leq |x-x_0| \leq R \}$ with inner radius $r$ and outer radius $R$,  the \emph{thickness ratio} of $A$ is defined to be the quantity $R/r$.

The following version of the Sobolev inequality is proved in \cite{carruth2023example}.
\begin{lem}\label{lem: sobolev}
    Let $\Omega\subset \R^2$ be a square, a ball, or an annulus with thickness ratio at most $C_0 \in [1,\infty)$ and let $1 < r< 2$. For any $F \in L^{2,r}(\Omega)$ and any $x \in \Omega$, we define an affine function $T_{x,\Omega}(F):\R^2 \rightarrow \R$ by
    \[
    T_{x,\Omega}(F)(y) = F(x) + (\nabla F)_{\Omega}\cdot (y-x).
    \]
    We then have, for any $y \in \R^2$, that
    \[
    |T_{x,\Omega}(F)(y) - T_{z, \Omega}(F)(y)|
        \lesssim_{r,C_0}  \|F\|_{L^{2,r}(\Omega)} |x-z|^{2-2/r}\;\text{for any}\; x,z \in \Omega.
    \]
    In particular, 
    \[
    \| F - T_{x,\Omega}(F)\|_{L^\infty(\Omega)} \lesssim_{r,C_0} \diam(\Omega)^{2-2/r}\|F\|_{L^{2,r}(\Omega)}.
    \]
\end{lem}

Let $B(z,r)$ denote the ball of radius $r>0$ centered at $z \in \R^2$, and let $\cM$ denote the uncentered Hardy-Littlewood maximal operator, i.e.,
\[
(\cM f)(x) = \sup_{B(z,r) \ni x} \frac{1}{|B(z,r)|} \int_{B(z,r)} f(y)\ dy\;\text{for any}\; f \in L^1_{\text{loc}}(\R^2).
\]
Recall that $\cM$ is a bounded operator from $L^q(\R^2)$ to $L^q(\R^2)$ for any $1 < q\le  \infty$ (see, e.g., \cite{stein1993harmonic}).

\section{Properties of the map $\Psi$}\label{sec: psi}

Recall from Section \ref{sec: intro} that the children of any vertex of $V$ are ordered. Precisely, for children $x, y$ of a common parent we say that $x < y$ if $x_{d(x)} < y_{d(y)}$.

This induces an ordering on the leaves $\partial V$. Consider distinct $v, w \in \partial V$ with $d(\lca(v,w)) = m$, so that $v_i = w_i$ for all $0 \le i \le m$ but $v_{m+1} \ne w_{m+1}$. Then we say that $v < w$ if and only if $v_{m+1} < w_{m+1}$.

\begin{lem}\label{lem: psi}
    Let $\Psi: V \rightarrow \R$ be the map defined in Section \ref{sec: intro}. Then for any distinct $v, w \in \partial V$, the following hold:
    \begin{itemize}
    \item If $v < w$, then $0 \leq \Psi(v) < \Psi(w) < 2$.
    \item For an absolute constant $K>1$,
        \[
K^{-1}W_{\emph{lca}(v,w)}/N \le |\Psi(w) - \Psi(v)| \le KW_{\emph{lca}(v,w)}.
        \]
    \end{itemize}
\end{lem}

\begin{proof}

We claim that $\Psi(w) \in [0,2)$ for any $w \in V$. Indeed, $\Psi(w) \geq 0$ is immediate from the representation \eqref{eq: psi 2}. Because the weights are radially decaying and $W_{\emptyset} = 1$, we have $W_{\tilde{v}} \leq \varepsilon^{d(\tilde{v})}$ for all $\tilde{v} \in V$. Observe that $\varepsilon < 1/2$, since we have assumed $\varepsilon < k_0/N$ for small enough $k_0$. By \eqref{eq: psi 2} and the fact that $w_i \leq N-1$ for all $i$, we deduce that
\[
\Psi(w) = \sum_{i=1}^{d(v)} W_{\pi_{i-1}(w)} \cdot \frac{w_i}{N-1} \leq \sum_{i=1}^{d(w) } W_{\pi_{i-1}(w)} \leq  \sum_{i=1}^{d(w)} \varepsilon^{i-1} < 2.
\]

We will show that the embedding $\Psi|_{\partial V} :\partial V\rightarrow \R$ is order preserving. To see this, we fix $v,w \in \partial V$ with $d(\lca(v,w)) = m$ and $v < w$. By \eqref{eq: psi 2}, we have
\begin{equation}\label{eq: prop 1}
\Psi(w) - \Psi(v) \ge W_{\pi_{m}(w)}\cdot \frac{w_{m+1} - v_{m+1}}{N-1} - \mathbbm{1}_{d(v) \ge m+2}\sum_{i = m+2}^{d(v)} W_{\pi_{i-1}(v)}\cdot \frac{v_i}{N-1}. 
\end{equation}
Note that $ w_{m+1} > v_{m+1}$, so  $w_{m+1} - v_{m+1} \geq 1$. If $d(v) = m+1$, then \eqref{eq: prop 1} implies that 
\[
\Psi(w) - \Psi(v) > \frac{W_{\pi_m(w)}}{N}.
\]
Assume instead that $d(v) \ge m+2$. From \eqref{eqn: rad_decay}, and since $v_i \le N-1$ and $\varepsilon < 1/2$, we get
\begin{equation}\label{eq: prop 3}
\sum_{i = m+2}^{d(v)} W_{\pi_{i-1}(v)}\cdot \frac{v_i}{N-1} \le W_{\pi_{m}(v)}\sum_{i = 1}^{d(v)-m-1} \varepsilon^{i} \leq 2 \varepsilon W_{\pi_{m}(v)}.
\end{equation}
Combining this with \eqref{eq: prop 1} and using that $\pi_m(w) = \pi_m(v) = \lca(w,v)$, and $\varepsilon \le k_0/N$ for sufficiently small $k_0$, gives
\begin{equation}\label{eq: prop 2}
\Psi(w) - \Psi(v) \ge W_{\lca(w,v)}\cdot \frac{1 - 2k_0}{N} > \frac{W_{\lca(w,v)}}{2N}.
\end{equation}
In particular, $\Psi(w) > \Psi(v)$ for any $v,w \in \partial V$ with $w > v$. Therefore, the embedding $\Psi|_{\partial V}$ of $\partial V$ into $\R$ is order preserving.

We now claim that
\begin{equation}\label{eq: prop 5}
 |\Psi(v) - \Psi(w)| \le KW_{\lca(v,w)}\;\text{for any distinct}\; v, w \in \partial V.
\end{equation}
Consider distinct $v,w \in \partial V$ with $d(\lca(v,w)) = m$. Combining \eqref{eq: psi 2} and the triangle inequality gives
\[
|\Psi(v) - \Psi(w)| \le \sum_{i = m+1}^{d(v)} W_{\pi_{i-1}(v)}\cdot \frac{v_i}{N-1} + \sum_{i = m+1}^{d(w)} W_{\pi_{i-1}(w)}\cdot \frac{w_i}{N-1}.
\]
Arguing as in \eqref{eq: prop 3}, we deduce \eqref{eq: prop 5}. Together with \eqref{eq: prop 2}, we have established the second bullet point of the lemma.

\end{proof}

Recall that the map $\Psi$ is used to define the set $E_2$ in \eqref{defn: E2}, and $E_1 \subset \R \times \{0\}$ is defined in \eqref{defn: E1}. We now establish some basic properties of the set $E$. 

\begin{lem}\label{lem: sep}
The set $E$ has the following properties:
    \begin{enumerate}
        \item $E \subset [0,2) \times [0,2)$\label{lem: sep1},
        \item $E$ is $\Delta$-separated,\label{lem: sep2}
        \item Let $x \in E_2$. Then
        \begin{enumerate}
            \item $\Delta \le x^{(2)} \le \dist(x, E_1) \le 2 x^{(2)}$,
            \item $\Delta \leq x^{(2)} \leq \dist(x, E_2\backslash \{x\})$.
        \end{enumerate}\label{lem: sep3}
    \end{enumerate}
\end{lem}

\begin{proof}
Since the weights are radially decaying, we have $W_v < 1$ for any $v \in \partial V$. Combining this with \eqref{defn: E1} and Lemma \ref{lem: psi}, we deduce Part \ref{lem: sep1} of the lemma.

Note that Part \ref{lem: sep2} of the lemma follows from Part \ref{lem: sep3} (recall that the points of $E_1$ are $\Delta$-separated by definition. 

It remains to prove Part \ref{lem: sep3}.

Let $x = (x^{(1)}, x^{(2)}) \in E_2$. Then $x = (\Psi(v), W_v)$ for some $v \in \partial V$. Therefore,
\[
\dist(x, E_1)  \ge x^{(2)}  \ge \min_{v \in \partial V} W_v = \Delta.
\]
Since $\Psi(v) \in [0,2)$ for $v \in \partial V$, and by definition of $E_1$ in \eqref{defn: E1},
\[
\dist(x,E_1) \leq x^{(2)} + \dist((x^{(1)},0),E_1) \leq x^{(2)} + \Delta \leq 2 x^{(2)}.
\]
By Lemma \ref{lem: psi}, and since the weights are radially decaying,
\[
\dist(x, E_2\backslash \{x\}) \ge \frac{W_{\pi(v)}}{KN} \ge \frac{W_v}{KN\varepsilon}.
\]
Recall that $\varepsilon \le k_0 / N$, and thus taking $k_0$ sufficiently small gives 
\[
\dist(x, E_2\backslash \{x\}) \ge \frac{x^{(2)}}{KN\varepsilon} \ge \frac{x^{(2)}}{Kk_0}  \ge  x^{(2)} \geq \Delta.
\]
This concludes the proof of Part \ref{lem: sep3}.
\end{proof}

\section{The Whitney decomposition}\label{sec: decomp}

This section borrows heavily from Section 3 of our previous paper \cite{carruth2023example}.

We will work with squares in $\R^2$; by this we mean axis parallel squares of the form $Q = [a_1, b_1) \times [a_2, b_2)$. We let $\delta_Q$ denote the sidelength of such a square $Q$. To \emph{bisect} a square $Q$ is to partition $Q$ into squares $Q_1, Q_2, Q_3 ,Q_4$, where $\delta_{Q_i} = \delta_Q / 2$ for each $i = 1,2,3,4$. We refer to the $Q_i$ as the \emph{children} of $Q$. 

We define a square $Q^0 = [-3, 5) \times [-3,5)$; note that $E \subset Q^0$. A \emph{dyadic} square $Q$ is one that arises from repeated bisection of $Q^0$. Every dyadic square $Q \neq Q^0$ is the child of some square $Q'$; we call $Q'$ the \emph{parent} of $Q$ and denote this by $(Q)^+ = Q'$.

We say that two dyadic square $Q, Q'$ \emph{touch} if $1.1Q \cap 1.1Q' \ne \emptyset$. We write $Q \touch Q'$ to denote that $Q$ touches $Q'$.  

For any dyadic square $Q$, we define a collection $\cW(Q)$, called the \emph{Whitney decomposition of} $Q$, by setting
\[
\cW(Q) = \{Q\} \;\text{if } \#(3Q\cap E) \le 1,
\]
and
\[
\cW(Q) = \bigcup \{\cW(Q') : (Q')^+ = Q\}\;\text{if}\; \#(3Q \cap E) \ge 2.
\]
We write $\cW = \cW(Q^0)$. Evidently, $\cW$ is a partition of $Q^0$ by dyadic squares. Note that $\cW \neq \{ Q^0\}$ because $\#(3Q^0 \cap E) = \# E \geq 2$. We now collect a few useful properties of the family $\cW$.

\begin{lem}\label{lem: cz}
    The collection $\cW$ has the following properties:
    \begin{enumerate}
    \item For any $Q \in \cW$, we have $\#(1.1Q \cap E) \le 1$ and $\#(3Q^+ \cap E ) \ge 2$.
        \item For any $Q, Q' \in \cW$ with $Q \touch Q'$, we have $\frac{1}{2}\delta_Q \le \delta_{Q'}\le 2 \delta_Q$.
        \item For any $Q \in \cW$, we have
        \[
        \#\{Q' : Q\touch Q'\} \lesssim 1.
        \]
        \item For any $x \in \R^2$,
        \[
        \#\{Q \in \cW : x \in 1.1Q\} \lesssim 1.
        \]
        \item For any $Q \in \cW$ with $\#(1.1Q \cap E) = 0$, we have $\delta_Q \approx \emph{dist}(Q,E)$.
    \end{enumerate}
\end{lem}
We omit the proof of Lemma \ref{lem: cz}, as this type of decomposition is standard in the literature; see, e.g., \cite{fefferman2020fitting}.

Observe that property 2 of Lemma \ref{lem: cz}, combined with the fact that all dyadic squares arise from repeated bisection of $Q^0$, implies that for any $Q, Q' \in \cW$ with $Q \leftrightarrow Q'$, we in fact have $\partial Q \cap \partial Q' \ne \emptyset$.

We now remark that
\begin{equation}\label{eq: cz 1}
\delta_Q \ge \frac{\Delta}{20}\;\text{for any}\; Q \in \cW.
\end{equation}
To see this, observe that $3Q^+ \subset 9Q$. Thus, Property 1 of Lemma \ref{lem: cz} implies that $\# (9Q \cap E) \ge 2$. Since the distance between distinct points of $E$ is at least $\Delta$, it follows that $\delta_Q \ge \Delta /20$, as claimed.

Let $\partial Q$ denote the boundary of a square $Q$. We say that $Q \in \cW$ is a \emph{boundary square} if $1.1Q \cap \partial Q^0 \ne \emptyset$. Denote the set of boundary squares by $\partial \cW$. We remark that since dyadic squares arise from repeated bisection of $Q^0$, any boundary square $Q \in \partial \cW$ satisfies the stronger property $Q \cap \partial Q^0 \ne \emptyset$. Observe that
\begin{equation}\label{eq: bdry 1}
    \delta_Q \ge 1\;\text{for any}\; Q \in \partial \cW.
\end{equation}
Indeed, this follows because $E \subset [0,2) \times [0,2)$, and if $Q$ is a dyadic square intersecting the boundary of $Q^0 = [-3,5) \times [-3,5)$ with $\delta_Q \leq 1/2$, then $Q^+$ is a dyadic square intersecting the boundary of $Q^0$ with $\delta_{Q^+} \leq 1$, which implies that $3Q^+$ is disjoint from $E$, and hence $Q \notin \cW$ (see Part 1 of Lemma 4).

Note that 
\begin{equation}\label{z0w0:prop}
E \subset 50 Q \mbox{ for any } Q \in \partial \cW.
\end{equation}
This follows from \eqref{eq: bdry 1} and because $E \subset [0,2) \times [0,2)$, while $Q \subset [-3,5) \times [-3,5)$.

\begin{defn}[Type I,II,II squares]
A square $Q \in \cW$ is of \emph{Type I} if $\#(1.1Q \cap E_1) = 1$, \emph{Type II} if $\#(1.1Q \cap E_2) = 1$, and \emph{Type III} if $\#(1.1Q \cap E) = 0$. The collections of squares of Type I, II, and III are denoted by $\cW_I$, $\cW_{II}$, and $\cW_{III}$, respectively. 
\end{defn}

The collections $\cW_I$, $\cW_{II}$, and $\cW_{III}$ form a partition of $\cW$ because $\#(1.1Q \cap E) \leq 1$ for any $Q \in \cW$, while the set $E$ is partitioned as $E  = E_1 \cup E_2$. Also observe that 
\[
\partial \cW \subset \cW_{III}.
\]

\begin{lem}\label{lem:dist_bd} For any $Q \in \cW$, we have $\delta_Q \approx (\Delta + \dist(Q, E_1))$.
\end{lem}
\begin{proof}
For $Q \in \cW_I$, we have $\#(1.1Q \cap E_1) = \#(3Q\cap E_1) = 1$. Since for each $x \in E_1$ there exists $y \in E_1$ such that $|x-y| = \Delta$, we deduce that $\delta_Q \lesssim \Delta$. Combining this with \eqref{eq: cz 1} gives
\begin{equation}\label{eq: distw1}
\delta_Q \approx \Delta \;\text{for any}\; Q \in\cW_I.
\end{equation}
Since any $Q \in \cW_I$ satisfies $1.1Q \cap E_1\ne \emptyset$, we also have $\dist(Q, E_1) \le \delta_Q$. This proves the lemma for $Q \in \cW_I$.

Note that for $Q \in \cW_{II}\cup\cW_{III}$ we have $\#(1.1Q \cap E_1) = 0$, and therefore 
\begin{equation}\label{eq: cz mid1}
\delta_Q \lesssim \dist(Q, E_1)\;\text{for any}\; Q \in \cW_{II}\cup\cW_{III}.
\end{equation}

For any $Q \in \cW$ we have $\#(3Q^+\cap E)\ge 2$ and thus $\#(9Q \cap E)\ge 2 $. If $9Q \cap E_1 \ne \emptyset$, then $\dist(Q, E_1) \lesssim \delta_Q$. Assume that $9Q\cap E_1 = \emptyset$. Then there are at least two distinct points in $9Q\cap E_2$; call them $v_Q, y_Q$. Since $v_Q, y_Q \in 9Q$, we have $\dist(v_Q, y_Q) \lesssim \delta_Q$. Using Part \ref{lem: sep3} of Lemma \ref{lem: sep}, we have
\[
\dist(v_Q, E_1) \approx v_Q^{(2)} \lesssim \dist(v_Q, y_Q) \lesssim \delta_Q,
\]
and therefore
\begin{equation*}
\dist(Q, E_1)\le \dist(Q, v_Q) + \dist(v_Q, E_1) \lesssim \delta_Q.
\end{equation*}
Combining this with \eqref{eq: cz mid1} proves that
\begin{equation}\label{eq: cz mid2}
\delta_Q \approx \dist(Q, E_1) \;\text{for any}\; Q \in \cW_{II}\cup\cW_{III};
\end{equation}
combining \eqref{eq: cz mid2} with \eqref{eq: cz 1} proves the lemma for $Q \in \cW_{II}\cup \cW_{III}$. This completes the proof of the lemma.
\end{proof}

\subsection{Basepoints}\label{sec: basepoint}

To each $x \in E_2$ we associate points $z_x, w_x \in E_1$ such that 
\begin{equation}\label{eq: cz 4.6}
\dist(x, E_1) = |x - z_x| \approx |x - w_x| \approx |z_x- w_x| \approx x^{(2)};
\end{equation}
this is possible thanks to Part \ref{lem: sep3}(a) of Lemma \ref{lem: sep} and the fact that the points of $E_1$ are equispaced in $[0,2) \times \{0\}$ with separation $\Delta$ (see \eqref{defn: E1}).

For each $Q \in \cW_{II}$ we let $x_Q$ be the unique point in $1.1Q \cap E = 1.1Q \cap E_2$. Note that $x_Q$ is undefined for $Q \in \cW \setminus \cW_{II}$.

We let $z_0 := (0,0)$ and $w_0 := (w_0^{(1)},0)$ be the points of maximal separation in $E_1$. Observe that $|z_0-w_0| \approx 1$. (See \eqref{defn: E1}.)

To each $Q \in \cW$ we associate a pair of points $z_Q,w_Q \in E_1$. We list the key properties of these points in the next lemma.

\begin{lem} \label{lem: bpoint}
There exists an absolute constant $K_0>1$ so that the following holds. For each $Q \in \cW$ there exist points $z_Q, w_Q \in K_0Q \cap E_1$, satisfying the conditions below.
\begin{enumerate}
    \item $|z_Q - w_Q| \approx \delta_Q$.
    \item If $Q \in \cW_I$, then $z_Q \in (1.1Q) \cap E_1$.
    \item If $Q \in \cW_{II}$, then $z_Q = z_{x_Q}$ and $w_Q = w_{x_Q}$.
    \item If $Q \in \partial \cW$, then $z_Q = z_0 $ and $w_Q = w_0$.
\end{enumerate}

\end{lem}

\begin{proof}
For each $Q \in \cW$ there exist points $z_Q,w_Q \in K_0Q \cap E_1$ satisfying $|z_Q - w_Q| \approx \delta_Q$ provided $K_0$ is sufficiently large; this is a consequence of Lemma \ref{lem:dist_bd} and the fact that the points of $E_1$ are equispaced in $[0,2) \times \{0\}$ with separation $\Delta$.

We make small modifications to this construction to establish conditions 2 -- 4 of the lemma.

If $Q \in \cW_I$, then instead select $z_Q \in 1.1Q \cap E_1$ and let $w_Q \in E_1$ be adjacent to $z_Q$ so that $|z_Q - w_Q| = \Delta \approx \delta_Q$ (see \eqref{eq: distw1}). Then $w_Q \in K_0 Q$ for $K_0$ sufficiently large. Consequently, $z_Q, w_Q \in K_0 Q \cap E_1$.

If $Q \in \cW_{II}$, then instead take $z_Q = z_{x_Q}$ and $w_Q = w_{x_Q}$, with $x_Q$ defined as above. By \eqref{eq: cz 4.6},
\[
|z_Q - w_Q| \approx |x_Q - z_Q| = \dist(x_Q, E_1).
\]
Because $x_Q \in 1.1Q$ and by \eqref{eq: cz mid2}, we have
\[
 \dist(x_Q, E_1) \lesssim \delta_Q + \dist(Q, E_1) \approx \delta_Q.
\]
Therefore, $|z_Q - w_Q| \approx |x_Q - z_Q | \lesssim \delta_Q$. Since $x_Q \in 1.1Q$, we deduce that $z_Q, w_Q \in K_0 Q$ for large enough $K_0$. Therefore, $z_Q, w_Q \in K_0 Q \cap E_1$, as claimed.

If $Q \in \partial \cW$ then we define $z_Q = z_0$ and $w_Q = w_0$, where $z_0 = (0,0)$ and $w_0 = (w_0^{(1)},0)$ are the leftmost and rightmost points of $E_1$. Note that $|z_Q - w_Q| \approx 1 \approx \delta_Q$. It follows from \eqref{z0w0:prop} that $z_0, w_0 \in 50 Q$, so, in particular (taking $K_0 \ge 50)$, $z_Q,w_Q \in K_0 Q \cap E_1$, as desired.
\end{proof}   

\subsection{Whitney partition of unity}

Let $\{\theta_Q\}_{Q\in \cW}$ be a partition of unity subordinate to $\cW$ constructed so that the following properties hold. For any $Q \in \cW$, 
\begin{enumerate}[label={\textbf{(POU\arabic*)}}, left=0pt]
    \item $\supp(\theta_Q) \subset 1.1Q$.
    \item For any $|\alpha| \le 2$, $\|\partial^\alpha \theta_Q\|_{L^\infty}\lesssim \delta_Q^{-|\alpha|}.$
    \item $ 0 \le \theta_Q \le 1$.
\end{enumerate}
For any $x \in Q^0$, 
\begin{enumerate}[label={\textbf{(POU4)}}, left=0pt]
\item $\sum_{Q \in \cW} \theta_Q(x) =1.$
\end{enumerate}

The construction of such a partition of unity is a standard exercise and may be found in the literature; e.g., see \cite{fefferman2020fitting}.

\begin{lem}[Patching Lemma]\label{lem: patching}
    \sloppypar Given affine polynomials $\{P_Q\}_{Q \in \cW}$, define ${F:Q^0 \rightarrow \R}$ by
\[
    F(x) = \sum_{ Q\in \cW} \theta_Q(x) P_Q(x).
    \]
    Then
    \[
    ||F||_{L^{2,p}(Q^0)}^p \lesssim_p \sum_{\substack{Q, Q' \in \cW:\\ Q\touch Q' }} ||P_Q - P_{Q'}||_{L^\infty(Q)}^p \delta_Q^{2-2p}.
    \]
\end{lem}
\begin{proof}
    Fix a square $Q' \in \cW$. Observe that 
    \[
    F(x) = \sum_{Q \in \cW} \theta_Q(x) [P_Q(x) - P_{Q'}(x)] + P_{Q'}(x) \qquad (x \in Q^0).
    \]
    By Property 4 of Lemma \ref{lem: cz}, there are a bounded number of squares $Q \in \cW$ for which $x \in (1.1Q) \cap Q'$. Therefore, by \textbf{(POU1)}, there are a bounded number of $Q \in \cW$ with $\supp(\theta_Q) \cap Q' \ne \emptyset$. Taking $2^{\text{nd}}$ derivatives, using \textbf{(POU2)}, and integrating $p^{\text{th}}$ powers then gives
    \[
    \|F\|_{L^{2,p}(Q')}^p \lesssim_p \sum_{\substack{Q\in \cW:\\ Q \touch Q'} } \big\{\delta_Q^{2 - 2p} \|P_Q - P_{Q'}\|^p_{L^\infty(Q)} + \delta_Q^{2-p}| \nabla (P_Q - P_{Q'})|^p\big\}.
    \]
    For any affine polynomial $P$ and any square $Q$, we have $|\nabla P| \le \delta_Q^{-1} || P||_{L^\infty(Q)}$, and thus
    \[
    \|F\|_{L^{2,p}(Q')}^p \lesssim_p \sum_{\substack{Q\in \cW:\\ Q \touch Q'} } \delta_Q^{2 - 2p} \|P_Q - P_{Q'}\|^p_{L^\infty(Q)}.
    \]
    Since $\cW$ is partition of $Q^0$, summing over $Q' \in \cW$ proves the lemma.
\end{proof}

\section{Clusters of the set $E_2$}\label{sec: cluster}

For the remainder of this article we fix a sufficiently large absolute constant $K_0$ so that the conclusion of Lemma \ref{lem: bpoint} holds. All constants $K, k,$ etc.\ may depend on $K_0$.

For each $v \in V$ we define the \emph{shadow}
\begin{equation}\label{eq: clust def}
S_v = \{ u \in \partial V :  \pi_{d(v)}(u) = v\}.
\end{equation}
Each shadow is a subset of $\partial V$; we let $\cS = \{S_v\}_{v\in V}$ be the collection of shadows. Recall that we defined
\begin{equation*}
E_2 = \{ (\Psi(v), W_v) : v \in \partial V\},
\end{equation*}
and therefore the set of leaves $\partial V$ is in one-to-one correspondence with the set $E_2$. This determines an injection $\cS \rightarrow 2^{E_2}$ (where $2^{E_2}$ denotes the power set of $E_2$). We define the \emph{cluster} $C_v\subset E_2$ to be the image of $S_v$ under this injection, i.e.,
\begin{equation}\label{eq: clust def2}
C_v = \{(\Psi(u), W_u) : u \in S_v\} \qquad (v\in V). 
\end{equation}

The set of all clusters 
\[
\cC := \{C_v \}_{v \in V}
\]
forms a tree under the relation of set inclusion, i.e., $C \in \cC$ is an ancestor of $C' \in \cC$ if $C' \subset C$. Observe that for any two clusters $C, C' \in \cC$ exactly one of the following is true: (1) $C$ is an ancestor or descendant of $C'$, or (2) $C \cap C' = \emptyset$. We identify this tree with the tree $V$ via the isomorphism $v \mapsto C_v$. As with $V$, we denote the set of leaves of $\cC$ by $\partial \cC = \{C_v\}_{v \in \partial V}$ and we write $\cC_0 = \cC \backslash \{C_{\emptyset}\}$ (note that $C_{\emptyset}$ is the root node of the tree $\cC$). 

We naturally associate to the tree $\cC$ a family of weights $\{W_C\}_{C \in \cC}$ by setting $W_{C_v} = W_v$ for every $v \in V$. We can then define the weighted Sobolev space $L^{1,p}(\cC)$ and the analogous trace space $L^{1,p}(\partial \cC)$. Since the weighted trees $V$ and $\cC$ are isomorphic, a bounded linear extension operator $H: L^{1,p}(\partial V) \rightarrow L^{1,p}(V)$ induces a bounded linear extension operator $\cH: L^{1,p}(\partial \cC) \rightarrow L^{1,p}(\cC)$, and vice versa. Moreover, such operators have equal operator norms. We will make use of these facts in Sections \ref{sec: plane} and \ref{sec: tree}.

We next detail some basic geometric properties of the clusters of $E_2$.

Note that the root of the tree $\cC$ is the set $C_{\emptyset}= E_2$, while the set of leaves $\partial \cC$ is in one-to-one correspondence with the singleton sets of $E_2$. Thus each $C \in \partial \cC$ is of the form $C = \{x_C\}$ for a unique point $x_C = (x_C^{(1)}, x_C^{(2)}) \in E_2$. Observe that
\begin{equation}\label{eq: b01}
    W_C = x_C^{(2)} \;\text{for every}\; C \in \partial \cC.
\end{equation}
Using Lemma \ref{lem: psi}, the definition of clusters (see \eqref{eq: clust def},  \eqref{eq: clust def2}), and the radial decay of the weights, we have
\begin{align}
    &N^{-1} W_C \lesssim \diam(C)  \lesssim W_C \;\text{for every}\; C \in \cC\backslash \partial \cC,\label{eq: b02}\\
    &\dist(C, C') \gtrsim N^{-1}(W_{\pi(C)} + W_{\pi(C')})\;\text{for any}\; C, C' \in \cC \;\text{with}\; C \cap C' = \emptyset.\label{eq: dis}
\end{align}

For each $C \in \cC$ we fix a point $y_C \in C$. Observe that the singleton cluster $\{y_C\} \subset E_2$  is contained in $C$. Thus, by \eqref{eq: b01}, and the radial decay of the weights,
\begin{equation}\label{eq: yCprop}
y_C^{(2)} = W_{\{y_C\}} \leq W_C.
\end{equation}

We let $\kappa > 10$ be a constant to be picked in a moment. Letting $B(x,r)\subset \R^2$ denote the ball of radius $r$ centered at $x$, we define
\begin{equation}\label{eq: bcdef}
B_C = B(y_C, \kappa K_1 W_C) \;\text{for every}\; C \in \cC.
\end{equation}
Here, $K_1>1$ is a fixed absolute constant chosen so that 
\begin{align}
    & C \subset \kappa^{-1} B_C\;\text{for every}\; C \in \cC,\label{eq: preb1} \\
    & Q^0 = [-3, 5)\times [-3, 5) \subset B_{E_2}\label{eq: preb3}
\end{align}
(see \eqref{eq: b02}); note that $K_1$ does not depend on $\kappa$.

To prove \eqref{eq: preb1}, note that if $C \in \partial C$ then $C$ is a singleton set, and $y_C$ is the unique point of $C$. But $y_C$ is the center of $B_C$, so  $C \subset \kappa^{-1} B_C$. On the other hand, if $C \in \cC \setminus \partial \cC$ then $\diam(C) \lesssim W_C$ by \eqref{eq: b02}. Note that $\kappa^{-1}B_C = B(y_C, K_1 W)$. Since $y_C \in C$, we have $C \subset \kappa^{-1} B_C$ if $K_1$ is large enough. 

To prove \eqref{eq: preb3}, recall that $C_\emptyset = E_2$ is the root of $\cC$ and we have normalized the weights of the tree so that $W_{E_2} = 1$. Then \eqref{eq: preb3} is immediate provided that $K_1$ is large enough.

Recall that the constant $K_0>1$ was fixed at the beginning of this section, and recall the assumption that $\varepsilon \leq k_0/N$ for a small enough constant $k_0$. We claim that the family of balls $\{B_C\}_{C \in \cC}$ has the following properties, provided $\kappa$ is a large enough constant and $k_0$ is sufficiently small depending on $\kappa$:
\begin{enumerate}[label={\textbf{(B\arabic*)}},left=0pt]
    \item $C \subset \kappa^{-1}B_C$ for every $C \in \cC$. \label{b1}
    \item $\kappa B_C \subset B_{\pi(C)}$ for every $C \in \cC_0$.\label{b2}
    \item $\diam(B_C) = 2K_1 \kappa W_C$ for every $C \in \cC$. \label{b3} 
    \item $\dist(K_0 B_C, K_0 B_{C'}) \gtrsim N^{-1}(W_{\pi(C)} + W_{\pi(C')})$ for any $C, C' \in \cC_0$ with $C\cap C' = \emptyset$.\label{b4}
\end{enumerate}
Properties \ref{b1} and \eqref{b3} follow from \eqref{eq: preb1} and \eqref{eq: bcdef}, respectively. We prove properties \ref{b2} and \ref{b4} in a moment. First, however, observe that property \ref{b4} implies:
\begin{enumerate}[label={\textbf{(B\arabic*)}},left=0pt]
\setcounter{enumi}{4}
    \item The collection $\{K_0 B_C\}_{C \in \partial\cC}$ is pairwise disjoint.\label{b5}
    \item For any $\ell \ge 0$ the collection
\[
\{K_0 B_C : C\in \cC, d(C) = \ell\}
\]
is pairwise disjoint (recall that $d(C)$ denotes the depth of a node $C$ in the tree $\cC$).\label{b6}
\end{enumerate}

(For the deduction of \ref{b6} from \ref{b4}, note that clusters of identical depth are not ancestors or descendents of each other, and hence, must be disjoint.)

We now prove property \ref{b2}. Let $C \in \cC_0$ and $y \in \kappa B_C$. Applying the triangle inequality, we get
\[
|y_{\pi(C)}-y| \le |y_{\pi(C)} - y_C| + |y_C - y|.
\]
Since $C \subset \pi(C)$, we have $y_{\pi(C)}, y_C \in \pi(C) \subset \kappa^{-1} B_{\pi(C)}$ due to \ref{b1}. Therefore, by \ref{b3},
\[
|y_{\pi(C)} - y_C| \le \kappa^{-1} \diam(B_{\pi(C)}) = 2K_1 W_{\pi(C)}.
\]
Similarly, since $y, y_C \in \kappa B_C$ we have
\[
|y_C - y| \le 2 \kappa^2 K_1 W_{C}.
\]
Combining this with the assumption of radially decreasing weights, we have
\begin{align*}
|y_{\pi(C)} - y| \le 2 K_1( W_{\pi(C)} + \kappa^2 W_C) \le 2 K_1 W_{\pi(C)} (1+ \kappa^2 \varepsilon ).
\end{align*}
Provided $\kappa \geq 4$ and $k_0 \leq 1/\kappa^2$, using that $\varepsilon \le k_0/N$, we deduce that
\[
|y_{\pi(C)} - y| \le \kappa K_1 W_{\pi(C)}.
\]
Because $y \in \kappa B_C$ is arbitrary, we have therefore shown that
\begin{equation*}
    \kappa B_C \subset B_{\pi(C)} \;\text{for any}\; C \in \cC_0,
\end{equation*}
proving \ref{b2}.

We now prove property \ref{b4}. Let $C, C' \in \cC$ with $C \cap C' = \emptyset$. Observe that $C \subset K_0 B_C$ and $C' \subset K_0 B_{C'}$. Hence,
\[
\begin{aligned}
    \dist(K_0 B_C, K_0B_{C'}) & \ge \dist(C,C') - \diam(K_0 B_C) - \diam(K_0 B_{C'})\\
    & = \dist(C,C') - 2\kappa K_0 K_1 (W_C + W_{C'}).
\end{aligned}
\]
Combining this with \eqref{eq: dis} and the assumption of radially decreasing weights, we have
\[
\dist(K_0 B_C, K_0 B_{C'}) \ge \frac{1}{N}(k-2N\varepsilon \kappa K_0 K_1) (W_{\pi(C)} + W_{\pi(C')})
\]
for an absolute constant $k > 0$. Recall that $\varepsilon \le k_0/N$. Thus, provided $k_0$ is sufficiently small depending on $\kappa$ we have
\begin{equation}\label{eq: distbs}
\dist(K_0 B_C, K_0 B_{C'}) \gtrsim N^{-1}(W_{\pi(C)} + W_{\pi(C')}).
\end{equation}
This concludes the proof of \ref{b4}.

Thanks to property \ref{b6} and \eqref{eq: preb3}, we can define a map $\cW \ni Q \mapsto C_Q \in \cC$ as follows: For $Q \in \cW$, we define $C_Q$ to be equal to the cluster $C \in \cC$ of maximum depth for which $Q \subset B_C$. The next lemma establishes some properties of this map.

\begin{lem}\label{lem: cluster}
Provided $\kappa$ is sufficiently large and $k_0$ is sufficiently small depending on $\kappa$, the map $Q \mapsto C_Q$ has the following properties:
\begin{enumerate}[label={\emph{\textbf{(\Alph*)}}}, ref={\textbf{(\Alph*)}}]
\item If $Q \in \cW_{II}$, then $C_Q = \{x_Q\}$. (Recall from Section \ref{sec: basepoint} that $x_Q$ is the unique point of $1.1 Q \cap E_2$ for $Q \in \cW_{II}$.) \label{clust a}
\item If $Q \in \partial \cW$, then $C_Q = E_2$.
    \item If $Q, Q' \in \cW$ with $Q \leftrightarrow Q'$ and $C_Q \ne C_{Q'}$, then either $C_Q = \pi(C_{Q'})$ or $C_{Q'} = \pi(C_Q)$.
    \item  Let $C \in \cC_0$ and define 
    \[
    \cQ_C = \{(Q,Q') \in \cW \times \cW : Q \leftrightarrow Q', C_Q = C, C_{Q'} = \pi(C)\}.
    \]
    Then
    \[
    \sum_{(Q,Q') \in \cQ_C} \delta_Q^{2-p} \lesssim_{p,\kappa} W_C^{2-p}.
    \]
\end{enumerate}
\end{lem}
\begin{proof}
Let $Q \in \cW_{II}$. Recall that the points $z_Q, w_Q$ were introduced in Lemma \ref{lem: bpoint}. For $Q \in \cW_{II}$, $z_Q = z_{x_Q}$ and $w_Q = w_{x_Q}$. By Part 1 of Lemma \ref{lem: bpoint} and \eqref{eq: cz 4.6},
\[
\delta_Q \approx |z_Q - w_Q| \approx x_Q^{(2)}.
\]
Combining this with \eqref{eq: b01} gives $\delta_Q \approx W_{\{x_Q\}}$ for every $ Q \in \cW_{II}$. Thus, \ref{b3} implies that 
\[
\diam(B_{\{x_Q\}}) \approx \kappa \delta_Q \;\text{for every}\; Q \in \cW_{II}.
\]
We have by \ref{b1} that $x_Q \in \kappa^{-1} B_{\{x_Q\}}$. Also, $x_Q \in 1.1 Q$. Therefore, for $\kappa$ large enough, we deduce that $Q \subset B_{\{x_Q\}}$. This proves \textbf{(A)}.

We claim that for any $Q \in \cW$  we have
    \begin{equation}\label{eq: dq size}
    \frac{\Delta}{20} \le \delta_Q \le 2K_1 \kappa W_{C_Q}.
    \end{equation}
The lower bound on $\delta_Q$ follows from \eqref{eq: cz 1}. The upper bound is a consequence of the fact that $Q \subset B_{C_Q}$ and \ref{b3}.

    By inequality \eqref{eq: bdry 1}, any $Q \in \partial \cW$ satisfies $\delta_Q \ge 1$. By \eqref{eq: dq size} and the radial decay of the weights, any $Q \in \cW$ with $C_Q \in \cC_0$ satisfies $\delta_Q \le 2 K_1 \kappa \varepsilon$. Since $\varepsilon \le k_0 /N$, provided $k_0$ is small enough depending on $\kappa$ we deduce that $C_Q \notin \cC_0$ for any $Q \in \partial \cW$. Therefore, $C_Q = E_2$ for $Q \in \partial \cW$, proving \textbf{(B)}.

    We now prove \textbf{(C)}. Suppose that $Q, Q' \in \cW$ with $Q \touch Q'$ and $C_Q \ne C_{Q'}$. By \ref{b4}, we must have $C_Q \cap C_{Q'} \ne \emptyset$ and thus either $C_Q \subset C_{Q'}$ or $C_{Q'} \subset C_Q$. Without loss of generality, assume that $C_Q \in \cC_0$ and $C_Q \subset C_{Q'}$. By Property 2 of Lemma \ref{lem: cz}, we have $\delta_Q \approx \delta_{Q'}$. Combining this with \eqref{eq: dq size} and \ref{b3}, we have
    \[
    \delta_{Q'} \lesssim \kappa W_{C_Q} \approx \diam(B_{C_Q}).
    \]
    Since $Q \touch Q'$ and $Q \subset B_{C_Q}$, we deduce that $Q' \subset K B_{C_Q}$ for an absolute constant $K$. If $\kappa > K$, then $K B_{C_Q} \subset \kappa B_{C_Q} \subset B_{\pi(C_Q)}$ due to \ref{b2}. Thus, $Q' \subset B_{\pi(C_Q)}$ and so $C_{Q'} \subset \pi(C_Q)$. Thus, we have shown that $C_Q \subsetneq C_{Q'} \subset \pi(C_Q)$. Therefore, $C_{Q'} = \pi(C_Q)$. This proves \textbf{(C)}.

    We now prove \textbf{(D)}. Fix $C \in\cC_0$. We claim that
    \begin{equation}\label{eq: a0}
\#\{(Q,Q')\in\cQ_C: \delta_Q = \delta\} \lesssim 1 \;\text{for every}\; \delta > 0.
\end{equation}
 Suppose $(Q,Q') \in \cQ_C$ with $ \delta_{Q}= \delta$.  Because $Q \subset B_C$, we have $\delta = \delta_Q \leq \diam(B_C)$. Because $C_{Q'} = \pi(C)$, it holds that $Q' \not\subset B_C$. Since $Q \subset B_C$ and $Q\leftrightarrow Q'$, it follows that $Q, Q'$ are contained in a $K \delta_Q$-neighborhood of the boundary of $B_C$ for an absolute constant $K$. By Lemma \ref{lem:dist_bd}, it is also the case that $Q, Q'$ are contained in a $K' \delta_Q$ neighborhood of the $x^{(1)}$-axis for another absolute constant $K'$. Therefore,
 \[
 \begin{aligned}
 &(Q,Q') \in \cQ_C, \; \delta_Q = \delta \implies \\
 & Q \subset \{ x \in \R^2 : \dist(x,\partial B_C) \leq K \delta \} \cap \{x \in \R^2 : |x^{(2)}| \leq K' \delta \}.
 \end{aligned}
 \]
One can verify from \eqref{eq: yCprop}, \eqref{eq: bcdef} that the Lebesgue measure of the region 
\[
\Omega(C,\delta) = \{ x \in \R^2 : \dist(x,\partial B_C) \leq K \delta \} \cap \{x \in \R^2 : |x^{(2)}| \leq K' \delta \}
\]
is upper bounded by $K'' \delta^2$ for any $\delta \leq \diam(B_C) = 2\kappa K_1 W_C$, for an absolute constant $K''$, provided $\kappa$ is sufficiently large. A simple packing argument then yields that the number of dyadic cubes $Q$ contained in $\Omega(C,\delta)$ with $\delta_Q = \delta$ is $\lesssim 1$. Note also for fixed $Q \in \cW$ as above, the number of $Q' \in \cW$ with $Q \leftrightarrow Q'$ is $\lesssim 1$ (see Lemma \ref{lem: cz}). This completes the proof of \eqref{eq: a0}.

 Combining \eqref{eq: dq size} and \eqref{eq: a0}, and using that $2-p>0$, we see that
  \begin{equation*}
    \sum_{(Q, Q') \in \cQ_C} \delta_Q^{2-p} \le \sum_{\ell \le \log_2(2K_1\kappa W_C)} \sum_{(Q,Q') \in \cQ_C, \delta_Q = 2^\ell} 2^{\ell(2-p)} \lesssim_{p,\kappa} W_C^{2-p}.
 \end{equation*}    
This completes the proof of \textbf{(D)}.
\end{proof}

For the remainder of the article we fix $\kappa > 10$ to be a large enough constant so that we can apply Lemma \ref{lem: cluster}, and we assume that $k_0$ is sufficiently small so that the conclusion of Lemma \ref{lem: cluster} holds.

Recall that every $C \in \partial \cC$ is of the form $C=\{x_C\}$ for a unique $x_C \in E_2$, and recall that the points $z_x, w_x \in E_1$  for $x \in E_2$ were defined in Section \ref{sec: basepoint}. Note that the points $\{x,z_x,w_x\}$ in $\R^2$ are not colinear because $x \in E_2 \subset \R \times \{\Delta\}$ and $z_x,w_x \in E_1 \subset \R \times \{0\}$.
\begin{lem}\label{lem: ball est}
    For any $G \in L^{2,p}(\R^2)$ and $x \in E_2$, let $T_x(G)$ denote the unique affine polynomial satisfying
\begin{equation}\label{eq: jet def}
T_x(G)|_{\{x, z_x, w_x\}} = G|_{\{x, z_x, w_x\}}.
\end{equation} For any $G\in L^{2,p}(\R^2)$, the following holds:
    \begin{multline*}
        \sum_{C \in \cC_0}|(\partial_2 G)_{B_C} - (\partial_2 G)_{B_{\pi(C)}}|^p \cdot W_C^{2-p} + \sum_{ C \in \partial \cC} |\partial_2(T_{x_C}(G)) - (\partial_2 G)_{B_C}|^p\cdot W_C^{2-p} \\
        \lesssim_{p,N} ||G||_{L^{2,p}(\R^2)}^p.
    \end{multline*}
\end{lem}

\begin{proof}
Let $C \in \partial \cC$. By \ref{b1} and \ref{b3} we have $x_C \in B_C$ and $\diam(B_C) \approx W_C$. For any $Q \in \cW_{II}$ with $x_Q = x_C$, we have $z_{x_C}, w_{x_C} \in K_0 Q$ by Lemma \ref{lem: bpoint}. By Part \ref{clust a} of Lemma \ref{lem: cluster}, we also have $Q \subset B_C$. Thus $z_{x_C}, w_{x_C} \in K_0 B_C$. By Lemma \ref{lem: sobolev}, we have
    \[
    |\partial_2(T_{x_C}(G)) - (\partial_2 G)_{B_C}|^p\cdot W_C^{2-p} \lesssim_p ||G||_{L^{2,p}(K_0B_C)}^p\;\text{for every}\; C \in \partial \cC.
    \]
    Thanks to \ref{b5}, the collection $\{K_0 B_C\}_{C \in \partial \cC}$ is pairwise disjoint. We conclude that
    \begin{equation}\label{eq: ball est 1}
    \sum_{ C \in \partial \cC} |\partial_2(T_{x_C}(G)) - (\partial_2 G)_{B_C}|^p\cdot W_C^{2-p}
        \lesssim_p ||G||_{L^{2,p}(\R^2)}^p.
    \end{equation}

    For $C \in \cC_0$, let $r_C$ denote the radius of the ball $B_C$ (i.e., $r_C = \kappa K_1 W_C$). Using \ref{b1}, \ref{b2}, \ref{b3}, and \ref{b4} we introduce a family of annuli $\{A_C\}_{C \in \cC_0}$ with the following properties:
    \begin{enumerate}
        \item $A_C$ is centered at $y_C$, has inner radius $r_C$, and has outer radius $10^{M_C+1}r_C$ for some integer $M_C 
        \geq 0$ such that $10^{M_C+1}r_C \approx_N W_{\pi(C)}$.
        \item $A_C \subset \frac{1}{2} B_{\pi(C)}$.
        \item The family $\{A_C\}_{C \in \cC_0}$ is pairwise disjoint.
    \end{enumerate}

    To define the annuli, observe by \eqref{eq: dis} that
    \begin{equation}\label{eq: dlb}
        \dist(C,C') \geq k N^{-1} (W_{\pi(C)} + W_{\pi(C')}) \mbox{ when } C,C' \in \cC_0, \; C \cap C' = \emptyset.
    \end{equation}
    for an absolute constant $k \in (0,1)$. We choose $M_C \geq 0$ to be the largest integer satisfying the inequality \begin{equation}\label{eq: MC}
        10^{M_C+1} r_C < (k/4) N^{-1} W_{\pi(C)}. 
    \end{equation}
    Recall that $r_C = \kappa K_1 W_C$. Therefore, the inequality admits a solution $M_C \geq 0$ provided $10 \kappa K_1 W_C < \frac{k}{4} N^{-1} W_{\pi(C)}$, which is satisfied provided $W_C < k' N^{-1} W_{\pi(C)}$ for an absolute constant $k'>0$. This is implied by the radial decay of the weights and the assumption that $\epsilon \leq k_0 N^{-1}$ for sufficiently small $k_0$. By the choice of $M_C$, $10^{M_C+1} r_C \approx_N W_{\pi(C)}$, verifying condition 1. 

Let $C \in \cC_0$. Observe that
\[
10^{M_C+1} r_C < (k/4) N^{-1} W_{\pi(C)} < (1/4) W_{\pi(C)} < (1/4)\kappa K_1 W_{\pi(C)} = r_{\pi(C)}/4.
 \]
According to \ref{b1}, and because $\kappa > 10$, we have 
\[
\diam(\pi(C)) \leq \kappa^{-1} \diam(B_{\pi(C)}) < \frac{1}{4} r_{\pi(C)}.
\]
Therefore,
\[
10^{M_C+1} r_C + \diam(\pi(C)) < r_{\pi(C)}/2.
\]
Since both $y_C,y_{\pi(C)} \in \pi(C)$, we have
\[
\begin{aligned}
A_C \subset B(y_C, 10^{M_C+1} r_C) &\subset B(y_{\pi(C)}, 10^{M_C+1} r_C + \diam(\pi(C))) \\
&\subset B(y_{\pi(C)},  r_{\pi(C)}/2) = (1/2)B_{\pi(C)},
\end{aligned}
\]
proving condition 2.
    
To verify condition 3, we fix $C, C' \in \cC_0$ with $C \neq C'$ and demonstrate that $A_C \cap A_{C'} = \emptyset$. Note that either $C \subset C'$, $C' \subset C$, or $C$ and $C'$ are disjoint. Suppose first  $C \subset C'$. Then also $C \subset \pi(C) \subset C'$, and according to condition 2, $A_C$ is contained in the interior of $B_{\pi(C)}$. Thanks to \ref{b2}, $B_{\pi(C)}\subset B_{C'}$, so that $A_C$ is contained in the interior of $B_{C'}$. Since $A_{C'}$ only intersects the boundary of $B_{C'}$, we conclude that $A_C \cap A_{C'} = \emptyset$. Similarly, $A_C \cap A_{C'} = \emptyset$ if $C' \subset C$. Finally, suppose $C \cap C' = \emptyset$. It follows from \eqref{eq: dlb}, \eqref{eq: MC} that
\[
B(y_C, 10^{M_C+1} r_C) \cap B(y_{C'}, 10^{M_{C'}+1} r_{C'}) = \emptyset.
\]
(Recall $y_C \in C$ and $y_{C'} \in C'$.) Hence, $A_C \cap A_{C'} = \emptyset$. This completes the proof of condition 3.

    For each $C \in \cC_0$ we define for each $0 \le \ell \le M_C$
    \[
    A_C^{(\ell)} = \{x \in \R^2 : 10^\ell r_C \le |x - y_C| \le 10^{\ell+1}r_C\}.
    \]
    Observe that
    \[
    A_C = \bigcup_{\ell=0}^{M_C} A_{C}^{(\ell)}.
    \]
    We define
    \[
    r = \frac{p+1}{2}
    \]
    and claim that for every $C\in \cC_0$ we have
\begin{equation}\label{eq: ball est 2}
    |(\partial_2 G)_{B_{\pi(C)}} - (\partial_2 G)_{B_C}|^p\cdot W_C^{2-p} \lesssim_{p,N} \| \cM( | \nabla^2 G|^r) \|_{L^{p/r}(A_C)}^{p/r}.
\end{equation}
Since the $A_C$ are pairwise disjoint, this implies that
\[
\sum_{C \in \cC_0} | (\partial_2 G)_{B_{\pi(C)}} - (\partial_2 G)_{B_C}|^p\cdot W_C^{2-p}\lesssim_{p,N} \| \cM(|\nabla^2 G|^r) \|_{L^{p/r}(\R^2)}^{p/r};
\]
we use the boundedness of the Hardy-Littlewood maximal operator from $L^{p/r}(\R^2)$ to $L^{p/r}(\R^2)$ to deduce that
\[
\sum_{C \in \cC_0} | (\partial_2 G)_{B_{\pi(C)}} - (\partial_2 G)_{B_C}|^p\cdot W_C^{2-p}\lesssim_{p,N} ||G||_{L^{2,p}(\R^2)}^p.
\]
Combining this with \eqref{eq: ball est 1} proves the lemma. We now prove \eqref{eq: ball est 2}.

Fix $C \in \cC_0$. By the Sobolev Inequality and the fact that $\diam(B_C) \approx \diam(A_C^{(0)}) \approx W_C$, we have
\begin{align*}
|(\partial_2 G)_{B_C} - (\partial_2 G)_{A_C^{(0)}}|  & \lesssim_p W_C^{1-2/r}\cdot ||G||_{L^{2,r}(B_C\cup A_{C}^{(0)})}\\
& \lesssim_p W_C \cdot (\cM(|\nabla^2 G|^r)(z))^{1/r}
\end{align*}
for any $z \in A_C^{(0)}$. Taking $p$-th powers and integrating over $A_C^{(0)}$ gives
\begin{equation}\label{eq: ann1}
|(\partial_2 G)_{B_C} - (\partial_2 G)_{A_C^{(0)}}|^p \cdot W_C^{2-p} \lesssim_p  ||\cM(|\nabla^2 G|^r)||_{L^{p/r}(A_{C}^{(0)})}^{p/r}.
\end{equation}

Similarly, we show that
\begin{equation}\label{eq: ann2}
\begin{aligned}
    |(\partial_2 G)_{B_{\pi(C)}} - (\partial_2 G)_{A_C^{(M_C)}}|^p & \cdot W_C^{2-p} \\
    & \lesssim_{p,N} 10^{M_C(p-2)}||\cM(|\nabla^2 G|^r)||_{L^{p/r}(A_C^{(M_C)})}^{p/r}
\end{aligned}
\end{equation}
and that for any $0 \le \ell < M_C$ we have
\begin{equation}\label{eq: ann3}
    |(\partial_2 G)_{A_C^{(\ell)}} - (\partial_2 G)_{A_C^{(\ell+1)}}|^p \cdot W_C^{2-p} \lesssim_p 10^{\ell(p-2)}||\cM(|\nabla^2 G|^r)||_{L^{p/r}(A_C^{(\ell)})}^{p/r}.
\end{equation}
For the first inequality above, we have used that $\diam(A_C^{(M_C)}) \approx 10^{M_C+1} r_C \approx_N W_{\pi(C)} \approx \diam(B_{\pi(C)})$ (see \eqref{b3}). We combine \eqref{eq: ann1}, \eqref{eq: ann2}, and \eqref{eq: ann3}, apply the triangle inequality and use that $2-p >0$ to get
\begin{align*}
|(\partial_2 G)_{B_C} - (\partial_2 G)_{B_{\pi(C)}}|^p\cdot W_C^{2-p} & \lesssim_{p,N} ||\cM(|\nabla^2 G|^r)||_{L^{p/r}(A_C)}^{p/r} \sum_{\ell = 0}^{M_C} 10^{\ell(p-2)}\\
& \lesssim_{p} ||\cM(|\nabla^2 G|^r)||_{L^{p/r}(A_C)}^{p/r}.
\end{align*}
This completes the proof of \eqref{eq: ball est 2}.

\end{proof}

\section{The extension operator for $L^{2,p}(\R^2)$}\label{sec: plane}

In this section, we assume the existence of a bounded linear extension operator $H: L^{1,p}(\partial V) \rightarrow L^{1,p}(V)$ as in Theorem \ref{thm: main}.

Recall from the previous section that the map $v \mapsto C_v$ is an isomorphism of the weighted trees $V$ and $\cC$. Therefore, via this isomorphism, $H$ induces a bounded linear extension operator $\cH: L^{1,p}(\partial \cC) \rightarrow L^{1,p}(\cC)$ satisfying
\[
||\cH||_{L^{1,p}(\partial \cC) \rightarrow L^{1,p}(\cC)}= ||H||_{L^{1,p}(\partial V) \rightarrow L^{1,p}(V)}.
\]
We will use $\cH$ to construct a bounded linear extension operator $T: L^{2,p}(E) \rightarrow L^{2,p}(\R^2)$ satisfying
\begin{equation}\label{eq: plane main}
||T||_{L^{2,p}(E)\rightarrow L^{2,p}(\R^2)} \lesssim_{p,N} ||\cH||_{L^{1,p}(\partial \cC)\rightarrow L^{1,p}(\cC)}.
\end{equation}
This proves one of the conditional statements in Theorem \ref{thm: main}; we prove the other in the next section.

We write $||\cH||:=||\cH||_{L^{1,p}(\partial \cC)\rightarrow L^{1,p}(\cC)}.$

Assume that we are given $f: E\rightarrow \R$. We will produce a function ${F \in L^{2,p}(Q^0)}$ satisfying:
\begin{enumerate}[label={\textbf{(F\arabic*)}},left=0pt]
    \item $F$ is determined linearly by the data $f$.\label{f1}
    \item $F|_E = f$.\label{f2}
    \item $||F||_{L^{2,p}(Q^0)} \lesssim_{p,N} ||\cH||\cdot ||G||_{L^{2,p}(\R^2)}$ for any $G \in L^{2,p}(\R^2)$ with $G|_E = f$.\label{f3}
\end{enumerate}
Once we produce such an $F$, it will be straightforward to extend it to a function defined on all of $\R^2$. Once we do this, we'll have constructed the operator $T$ introduced above.

The function $F$ has the form
\begin{equation}\label{eq: plane 1}
F(x) = \sum_{ Q \in \cW} P_Q(x) \theta_Q(x),
\end{equation}
where $\{\theta_Q\}_{Q \in \cW}$ is the partition of unity introduced in Section \ref{sec: decomp} and $\{P_Q\}_{Q \in \cW}$ is a family of affine polynomials, to be constructed in a moment.

First, recall that in Section \ref{sec: decomp} we associated to each $x \in E_2$ points $z_x, w_x \in E_1$ (see \eqref{eq: cz 4.6}. For every $x \in E_2$ we define $P_x$ to be the unique affine polynomial satisfying
\begin{equation}\label{eq: plane 5}
P_x|_{\{x, z_x, w_x\}} = f|_{\{x, z_x, w_x\}}.
\end{equation}

Recall that every $C \in \partial \cC$ is of the form $C = \{x_C\}$ for some $x_C \in E_2$. We can therefore define a function $\phi: \partial \cC \rightarrow \R$ by setting 
\begin{equation}\label{eq: plane 2}
\phi(C) = \partial_2 P_{x_C}\;\text{for } C \in \partial \cC.
\end{equation}
We now use the bounded linear extension operator $\cH$ to extend the function $\phi: \partial \cC \rightarrow \R$ to a function $\Phi: \cC \rightarrow \R$, i.e., we define
\[
\Phi(C) = \cH\phi(C) \;\text{for } C \in \cC.
\]

Recall that every $Q \in \cW$ is associated with
\begin{itemize}
    \item points $z_Q, w_Q \in K_0 Q \cap E_1$ satisfying $|z_Q - w_Q| \approx \delta_Q$ (see Lemma \ref{lem: bpoint}),
    \item a cluster $C_Q \in \cC$ (see Section \ref{sec: cluster}).
\end{itemize}
For every $Q \in \cW$, we define $L_Q$ to be the unique affine polynomial satisfying
\[
L_Q|_{\{z_Q, w_Q\}} = f|_{\{z_Q, w_Q\}}, \quad (\partial_2 L_Q) = 0.
\]
We are now ready to define the polynomials $P_Q$ introduced above.

We define (for $x = (x^{(1)}, x^{(2)}) \in \R^2$)
\begin{equation}\label{eq: pq def}
P_Q(x) = L_Q(x) + x^{(2)} \cdot \Phi(C_Q)\;\text{for } Q \in \cW.
\end{equation}
Now that we have defined the polynomials $P_Q$, our alleged interpolant $F$ is defined by \eqref{eq: plane 1}.

It is evident that $F$ satisfies condition \ref{f1}, thanks to the linearity of the operator $\cH$ and the definition of the polynomials $P_x$, $L_Q$.

By Lemma \ref{lem: bpoint}, for every $Q \in \cW_I$ the point $z_Q \in E_1$ is the unique point in  $1.1Q \cap E$. Since $E_1 \subset \R \times \{0\}$, it follows that
\[
P_Q|_{1.1Q \cap E} = L_Q|_{1.1Q \cap E} = f|_{1.1Q \cap E}\;\text{for every}\; Q \in \cW_I,
\]
and thus
\begin{equation*}
    F|_{E_1} = f|_{E_1}.
\end{equation*}

Let $Q \in \cW_{II}$ and recall that we write $x_Q$ to denote the unique point in $1.1Q \cap E_2$. By Lemma \ref{lem: bpoint}, we have in this case that $z_Q = z_{x_Q}$, $w_Q = w_{x_Q}$. Since $z_{x_Q}, w_{x_Q} \in E_1 \subset \R \times \{0\}$, we get
\[
P_Q|_{\{z_Q, w_Q\}} = f|_{\{z_Q, w_Q\}}= P_{x_Q}|_{\{z_Q, w_Q\}}.
\]
By Lemma \ref{lem: cluster}, we have $C_Q = \{x_Q\}\in \partial \cC$ for any $Q \in \cW_{II}$. Therefore
\[
\partial_2 P_Q = \Phi(C_Q) = \phi(C_Q) = \partial_2 P_{x_Q};
\]
it follows that
\begin{equation*}
P_Q = P_{x_Q}\; \text{for every } Q \in \cW_{II}.
\end{equation*}
Combining this with \eqref{eq: plane 5} gives
\[
P_Q|_{1.1Q \cap E} = P_{x_Q}|_{1.1Q \cap E} = f|_{1.1Q \cap E}\;\text{for every}\; Q \in \cW_{II},
\]
and thus
\[
F|_{E_2} = f|_{E_2}.
\]
We deduce that $F$ satisfies condition \ref{f2}.

We now prove that $F$ satisfies condition \ref{f3}. For any $G \in L^{2,p}(\R^2)$ with $G|_E = f$, we must show that
\begin{equation}\label{eq: plane 10}
||F||_{L^{2,p}(Q^0)}\lesssim_{p,N}  ||\cH||\cdot||G||_{L^{2,p}(\R^2)}.
\end{equation}

We fix such a $G$. By Lemma \ref{lem: patching}, we have
\[
||F||_{L^{2,p}(Q^0)}^p \lesssim_p \sum_{Q \leftrightarrow Q'} || P_Q - P_{Q'}||_{L^\infty(Q)}^p \delta_Q^{2-2p}.
\]
Combining this with the definition of $P_Q$ and using Lemma \ref{lem:dist_bd} gives
\begin{equation}\label{eq: plane 9}
\begin{split}
||F||&_{L^{2,p}(Q^0)}^p \\
&\lesssim_p \sum_{Q \leftrightarrow Q'} \Big\{ ||L_Q - L_{Q'}||_{L^\infty(Q)}^p \cdot \delta_Q^{2-2p} + | \Phi(C_Q) - \Phi(C_{Q'})|^p \cdot \delta_Q^{2-p}\Big\}.
\end{split}
\end{equation}

As in the previous section, we let
\[
r = \frac{p+1}{2}
\]
and claim that for any $Q, Q' \in \cW$ with $Q\touch Q'$ we have
\begin{equation}\label{eq: lq1.1}
\| L_Q - L_{Q'}\|_{L^\infty(Q)}^r \lesssim_p \delta_Q^{2r-2} \|G\|_{L^{2,r}(5 K_0 Q)}^r.
\end{equation}
This implies that
\[
\| L_Q - L_{Q'}\|_{L^\infty(Q)}^r \lesssim_p \delta_Q^{2r} \cM(|\nabla^2 G|^r)(z)\;\text{for any}\; z \in Q,
\]
where $\cM$ is the Hardy-Littlewood maximal operator (see Section \ref{sec: prelim}). Taking $(p/r)$-th powers and integrating gives
\[
\| L_Q - L_{Q'}\|_{L^\infty(Q)}^p \lesssim_p \delta_Q^{2p-2} \|\cM(|\nabla^2 G|^r)\|_{L^{p/r}(Q)}^{p/r}.
\]
By Property 3 of Lemma \ref{lem: cz}, we deduce that
\[
\sum_{Q \touch Q'} \|L_Q - L_{Q'}\|_{L^\infty(Q)}^p \delta_Q^{2-2p} \lesssim_p \sum_{Q \in \cW} \|\cM(|\nabla^2 G|^r)\|_{L^{p/r}(Q)}^{p/r}.
\]
Squares in $\cW$ have pairwise disjoint interiors (since $\cW$ is a partition of $Q^0$), and thus we have
\[
\sum_{Q \in \cW} \|\cM(|\nabla^2 G|^r)\|_{L^{p/r}(Q)}^{p/r} \le \| \cM(|\nabla^2 G|^r)\|_{L^{p/r}(\R^2)}^{p/r}.
\]
Because the Hardy-Littlewood maximal operator is bounded from $L^{q}(\R^2)$ to $L^q(\R^2)$ for $1 < q \le \infty$, we deduce that
\begin{equation}\label{eq: plane 12}
    \sum_{Q \touch Q'} ||L_Q - L_{Q'}||^p_{L^\infty(Q)}  \delta_Q^{2-2p} \lesssim_p ||G||_{L^{2,p}(\R^2)}^p\le ||\cH||^p\cdot ||G||_{L^{2,p}(\R^2)}^p.
\end{equation}
(The last inequality simply uses that $||\cH||\ge 1$.) We now prove \eqref{eq: lq1.1}.

Fix $Q, Q' \in \cW$ with $Q\touch Q'$ and observe that for any $x \in Q$ we have
\begin{equation*}
L_Q(x) - L_{Q'}(x) = f(z_Q) + (\nabla L_Q)\cdot (x-z_Q) - f(z_{Q'}) - (\nabla L_{Q'})\cdot (x- z_{Q'}).
\end{equation*}
Since $G(z_{Q'}) = f(z_{Q'})$, and since $z_Q^{(2)}= z_{Q'}^{(2)} = 0$, we have
\[
T_{z_{Q'},5K_0Q}(G)(z_Q) = f(z_{Q'}) + (\partial_1 G)_{5K_0Q}(z_Q^{(1)}-z_{Q'}^{(1)}).
\]
(See Section \ref{sec: prelim} for the definition of $T_{z_{Q'},5K_0Q}$.) By Lemma \ref{lem: cz}, we have $|x-z_Q|, |x-z_{Q'}|, |z_Q - z_{Q'}| \lesssim \delta_Q$ for $x \in Q$. Therefore, since $\partial_2 L_Q = \partial_2 L_{Q'}=0$, by the triangle inequality we have
\begin{equation}\label{eq: L1}
\begin{split}
    \|L_Q - L_{Q'}\|_{L^\infty(Q)} \le \; & | f(z_Q) - T_{z_{Q'}, 5K_0Q}(G)(z_Q)| \\ &+ |(\partial_1 G)_{5K_0Q} - \partial_1 L_Q|\cdot\delta_Q \\ 
    &+ |(\partial_1 G)_{5K_0Q} - \partial_1 L_{Q'}|\cdot \delta_Q.
\end{split}
\end{equation}
By Lemma \ref{lem: cz}, $\frac{1}{2}\delta_{Q'} \le \delta_Q \le 2 \delta_{Q'}$, and thus, $K_0Q' \subset 5K_0 Q$. In particular, 
\[
\{z_Q, w_Q, z_{Q'}, w_{Q'}\} \subset 5K_0 Q\cap E_1.
\]
Since $G|_E = f$,  Lemma \ref{lem: sobolev} implies
\begin{align*}
    &|f(z_Q) - T_{z_{Q'}, 5K_0 Q}(G)(z_Q)| \lesssim_p \delta_Q^{2-2/r}\|G\|_{L^{2,r}(5K_0 Q)},\\
    &|f(z_Q) - f(w_Q) - (\partial_1 G)_{5K_0 Q} \cdot(z_Q^{(1)} - w_Q^{(1)})| \lesssim_p \delta_Q^{2-2/r}\|G\|_{L^{2,r}(5K_0 Q)},\\
    & |f(z_{Q'}) - f(w_{Q'}) - (\partial_1 G)_{5K_0 Q}\cdot(z_{Q'}^{(1)} - w_{Q'}^{(1)})| \lesssim_p \delta_Q^{2-2/r}\|G\|_{L^{2,r}(5K_0 Q)}.
\end{align*}
By definition of $L_Q$ we have $\partial_1 L_Q = \frac{f(z_Q) - f(w_Q)}{z_Q^{(1)} - w_Q^{(1)}}$ and $\partial_1 L_{Q'} = \frac{f(z_{Q'}) - f(w_{Q'})}{z_{Q'}^{(1)} - w_{Q'}^{(1)}}$. Thus, by combining the previous inequality with \eqref{eq: L1} we deduce \eqref{eq: lq1.1}.

Next, we claim that
\begin{equation}\label{eq: plane 11}
\sum_{Q \touch Q'} |\Phi(C_Q) - \Phi(C_{Q'})|^p\cdot \delta_Q^{2-p} \lesssim_{p,N} ||\cH||^p\cdot ||G||_{L^{2,p}(\R^2)}^p.
\end{equation}
Combining this with \eqref{eq: plane 9} and \eqref{eq: plane 12} proves \eqref{eq: plane 10}, establishing \ref{f3}. We now prove \eqref{eq: plane 11}.

By Lemma \ref{lem: cluster}, if $Q \touch Q'$ and $C_Q \ne C_{Q'}$, then either $C_Q = \pi(C_{Q'})$ or $C_{Q'} = \pi(C_Q)$. Therefore,
\begin{equation}\label{eq: plane 3}
\sum_{Q \touch Q'}| \Phi(C_Q) - \Phi(C_{Q'})|^p \cdot \delta_Q^{2-p} \lesssim \sum_{ C \in \cC_0} | \Phi(C) - \Phi(\pi(C))|^p \sum_{(Q, Q') \in \cQ_C} \delta_Q^{2-p},
\end{equation}
where
\[
    \cQ_C = \{(Q,Q') \in \cW \times \cW : Q \leftrightarrow Q', C_Q = C, C_{Q'} = \pi(C)\}
\]
(as in Lemma \ref{lem: cluster}). Applying Lemma \ref{lem: cluster} gives
\[
\sum_{Q \touch Q'}| \Phi(C_Q) - \Phi(C_{Q'})|^p \cdot \delta_Q^{2-p} \lesssim_p \sum_{ C \in \cC_0} | \Phi(C) - \Phi(\pi(C))|^p \cdot W_C^{2-p}.
\]
Since $\cH$ is a bounded linear extension operator and $\Phi = \cH \phi$, we have that
\[
\sum_{C \in \cC_0}  |\Phi(C) - \Phi(\pi(C))|^p \cdot W_C^{2-p} \lesssim_p ||\cH||^p\cdot \sum_{C \in \cC_0 } |\Xi(C) - \Xi(\pi(C))|^p \cdot W_C^{2-p}
\]
for any $\Xi: \cC \rightarrow \R$ satisfying $\Xi|_{\partial \cC} = \phi$. Taking $\Xi(C) = (\partial_2 G)_{B_C}$ for $C \in \cC \backslash \partial \cC$, we use \eqref{eq: plane 2} and apply the triangle inequality to get
\begin{equation}\label{eq: plane 6}
\begin{split}
\sum_{C \in \cC_0} |\Phi(C) - \Phi(\pi(C))|^p \cdot W_C^{2-p}&\\ \lesssim_p  ||\cH||^p\cdot\bigg\{ &\sum_{C \in \cC_0 } |(\partial_2 G)_{B_C} - (\partial_2 G)_{B_{\pi(C)}}|^p \cdot W_C^{2-p} \\ &+ \sum_{C \in \partial\cC} |\partial_2 P_{x_C} - (\partial_2 G)_{B_C}|^p\cdot  W_C^{2-p} \bigg\}.
\end{split}
\end{equation}
Applying Lemma \ref{lem: ball est} (note that by \eqref{eq: plane 5}, \eqref{eq: jet def}, we have for every $x \in E_2$ that $P_x = T_x(G)$) gives
\begin{equation}\label{eq: plane 13}
\sum_{C \in \cC_0}  |\Phi(C) - \Phi(\pi(C))|^p \cdot W_C^{2-p} \lesssim_{p,N} ||\cH||^p\cdot||G||_{L^{2,p}(\R^2)}^p;
\end{equation}
we deduce \eqref{eq: plane 11}.

We have thus constructed $F\in L^{2,p}(Q^0)$ and shown that it satisfies \ref{f1}--\ref{f3}. It remains to extend $F$ to a function on all of $\R^2$.

Recall the set of boundary squares
\[
\partial \cW = \{Q \in \cW : 1.1 Q \cap \partial Q^0 \ne \emptyset\},
\]
introduced in Section \ref{sec: decomp}.

By Lemma \ref{lem: bpoint}, there exist points $z_0, w_0$ so that
\[
z_Q = z_0, \; w_Q = w_0\; \text{for all}\; Q \in \partial \cW.
\]
We define $L_0$ to be the unique affine polynomial satisfying
\[
L_0|_{\{z_0, w_0\}} = f|_{\{z_0, w_0\}} \quad \text{and}\quad \partial_2 L_0 = 0.
\]
We then have
\[
L_Q = L_0\; \text{for all}\; Q \in \partial \cW.
\]
Similarly, by Lemma \ref{lem: cluster} we have for every $Q \in \partial \cW$ that $C_Q = E_2$ and thus
\[
\Phi(C_Q) = \Phi(E_2)\;\text{for all}\; Q \in \partial \cW.
\]
Invoking \eqref{eq: pq def}, we see that
\[
P_Q = L_0 + x^{(2)}\cdot \Phi(E_2)\;\text{for all} \; Q \in \partial \cW.
\]
We define
\[
\tilde{F}(x) = \begin{cases}
F(x) & \text{if } x \in Q^0,\\
L_0(x) + x^{(2)}\cdot\Phi(E_2) &\text{if } x \notin Q^0.
\end{cases}
\]
Recall (see \eqref{eq: plane 1}) that $F$ is of the form
\[
F = \sum_{Q\in\cW} P_Q \theta_Q,
\]
with $\supp(\theta_Q) \subset 1.1Q$, and therefore $\tilde{F} \in C^2(\R^2)$. Clearly, then, $||\tilde{F}||_{L^{2,p}(\R^2)} = ||F||_{L^{2,p}(Q^0)}$ and $\tilde{F}|_E = F|_E = f$. Moreover, $\tilde{F}$ depends linearly on $f$ thanks to \ref{f1}, the definition of $L_0$, and the definition of $\Phi$.

For any $f \in L^{2,p}(E)$, we define
\[
(Tf)(x) = \tilde{F}(x) \;(x \in \R^2).
\]
Then $T: L^{2,p}(E) \rightarrow L^{2,p}(\R^2)$ is a bounded linear extension operator satisfying \eqref{eq: plane main}.

\section{The extension operator for $L^{1,p}(V)$}\label{sec: tree}

In this section, we assume the existence of a bounded linear extension operator $T: L^{2,p}(E) \rightarrow L^{2,p}(\R^2)$ as in Theorem \ref{thm: main}. Using $T$, we construct a bounded linear extension operator $\cH : L^{1,p}(\partial \cC) \rightarrow L^{1,p}(\cC)$ satisfying
\begin{equation}\label{eq: 7main}
||\cH||_{L^{1,p}(\partial \cC) \rightarrow L^{1,p}(\cC)} \lesssim_{p,N} ||T||_{L^{2,p}(E)\rightarrow L^{2,p}(\R^2)}.
\end{equation}
As in the previous section, we use that $V$ and $\cC$ are isomorphic to observe that $\cH$ induces a bounded linear extension operator $H: L^{1,p}(\partial V) \rightarrow L^{1,p}(V)$ satisfying
\[
||H||_{L^{1,p}(\partial V) \rightarrow L^{1,p}(V)} = ||\cH||_{L^{1,p}(\partial \cC) \rightarrow L^{1,p}(\cC)}.
\]
Combined with the results of the previous section, this proves Theorem \ref{thm: main}.

We write $||T|| := ||T||_{L^{2,p}(E)\rightarrow L^{2,p}(\R^2)}$.

Suppose that we are given a function $\phi: \partial \cC \rightarrow \R$. We must then construct a function $\Phi: \cC \rightarrow \R$ satisfying:
\begin{enumerate}[label={\textbf{($\Phi$\arabic*)}}, left=0pt]
    \item $\Phi$ is determined linearly by the data $\phi$,\label{phi1}
    \item $\Phi|_{\partial \cC} = \phi$,\label{phi2}
    \item $||\Phi||_{L^{1,p}(\cC)} \lesssim_{p,N} ||T||\cdot ||\Xi||_{L^{1,p}(\cC)}$ for any $\Xi :\cC\rightarrow \R$ with $\Xi|_{\partial \cC} = \phi$.\label{phi3}
\end{enumerate}
Once we've constructed such a $\Phi$, we define
\[
\cH \phi (C) = \Phi(C)\;\text{for all} \; C \in \cC,
\]
establishing \eqref{eq: 7main}. We prepare to construct the function $\Phi$.

First, recall that for every $x \in E_2$ we have $\{x\} \in \partial \cC$. We can thus define a function $f: E \rightarrow \R$ by setting
\[
f(x) = \begin{cases}
    0 &\text{if } x \in E_1,\\
    \phi(\{x\})\cdot W_{\{x\}} &\text{if } x \in E_2.
\end{cases}
\]
We then apply the extension operator $T$ to $f$ and obtain a function $F \in L^{2,p}(\R^2)$, i.e., we define
\begin{equation}\label{eq: F def}
F(x) = Tf(x)\;\text{for } x\in \R^2.
\end{equation}
We now define the function $\Phi: \cC \rightarrow \R$ by
\[
\Phi(C) = \begin{cases}
    \phi(C) &\text{if } C \in \partial \cC,\\
    (\partial_2 F)_{B_C} &\text{if } C \in \cC \backslash \partial \cC.
\end{cases}
\]
Property \ref{phi2} is an immediate consequence of the definition of $\Phi$, and Property \ref{phi1} follows easily from the linearity of $T$ and the definition of $f$. It remains to prove that $\Phi$ satisfies \ref{phi3}.

Observe that
\begin{align*}
||\Phi||_{L^{1,p}(\cC)}^p = & \sum_{C \in \partial \cC} | \phi(C) - (\partial_2 F)_{B_{\pi(C)}}|^p\cdot W_C^{2-p} \\
&+ \sum_{ C \in \cC_0 \backslash \partial \cC} | (\partial_2 F)_{B_C} - (\partial_2 F)_{B_{\pi(C)}}|^p\cdot W_C^{2-p}.
\end{align*}
Applying the triangle inequality gives
\begin{align*}
||\Phi||_{L^{1,p}(\cC)}^p =& \sum_{C \in \cC_0} | (\partial_2 F)_{B_C} - (\partial_2 F)_{B_{\pi(C)}}|^p \cdot W_C^{2-p} \\
&+ \sum_{ C\in \partial \cC} |\phi(C) - (\partial_2 F)_{B_{C}}|^p\cdot W_C^{2-p}.
\end{align*}
Observe that $\phi(C) = \partial_2 (T_{x_C}(F))$ for $C = \{x_C\} \in \partial C$ (see \eqref{eq: jet def} and the definition of $f$); we thus apply Lemma \ref{lem: ball est} to deduce that 
\[
||\Phi||_{L^{1,p}(\cC)}^p \lesssim_{p,N} ||F||_{L^{2,p}(\R^2)}^p.
\]
Since $T$ is a bounded linear extension operator, we have by \eqref{eq: F def} that
\[
||\Phi||_{L^{1,p}(\cC)}^p \lesssim_{p,N} ||T||^p\cdot ||G||_{L^{2,p}(\R^2)}^p
\]
for any $G \in L^{2,p}(\R^2)$ satisfying $G|_E = f$.

We now assume that we are given $\Xi : \cC \rightarrow \R$ satisfying $\Xi |_{\partial \cC} = \phi$. In a moment, we will define a function $\widetilde{G} \in L^{2,p}(\R^2)$ satisfying $\widetilde{G}|_E = f$ and
\begin{equation}\label{eq: tree 3}
||\widetilde{G}||_{L^{2,p}(\R^2)}^p \lesssim_p || \Xi ||_{L^{1,p}(\cC)}^p.
\end{equation}
Once we establish this estimate, we will have shown that $\Phi$ satisfies Property \ref{phi3}.

We define the function $\widetilde{G}$ introduced above using the Whitney Decomposition $\cW$ of $E$ (see Section \ref{sec: decomp}). First, we will define local affine interpolants $P_Q$ for each $Q \in \cW$ and set
\[
G = \sum_{Q \in \cW} P_Q \theta_Q \quad \mbox{on } Q^0,
\]
where $\theta_Q$ is the partition of unity introduced in Section \ref{sec: decomp}. 

We now define the $P_Q$. For each $C \in \cC$, we define the affine polynomial
\[
P_C(x^{(1)}, x^{(2)}) = x^{(2)}\cdot \Xi(C).
\]
Recall that in Section \ref{sec: cluster}, we associated to every $Q \in \cW$ a cluster $C_Q \in \cC$. We define
\[
P_Q = P_{C_Q} \;\text{for every}\; Q \in \cW.
\]

Since $P_Q|_{E_1} = 0$ for every $Q \in \cW$, we clearly have $G|_{E_1} = f|_{E_1} = 0$. Moreover, recall that for every $Q \in \cW_{II}$ we have $C_Q = \{x_Q\}\in \partial \cC$, where $x_Q = (x_Q^{(1)}, W_{\{x_Q\}})$ is the unique point contained in $1.1Q \cap E_2$.  Thus,
\[
P_Q(x_Q) = \phi(\{x_Q\})\cdot W_{\{x_Q\}} = f(x_Q)\;\text{for every}\; Q \in \cW_{II},
\]
and so $G|_{E_2} = f|_{E_2}$. Therefore, $G|_E = f$, as claimed. Applying Lemma \ref{lem: patching} to $G$ (and using Lemma \ref{lem:dist_bd}) gives
\begin{align*}
||G||_{L^{2,p}(Q^0)}^p \lesssim_p \sum_{Q \leftrightarrow Q'} |\Xi(C_Q) - \Xi(C_{Q'})|^p \delta_Q^{2-p}.
\end{align*}
By Lemma \ref{lem: cluster}, if $Q \touch Q'$ and $C_Q \ne C_{Q'}$, then we either have $C_Q = \pi(C_{Q'})$ or $C_{Q'} = \pi(C_Q)$. Thus
\[
\sum_{Q \leftrightarrow Q'} |\Xi(C_Q) - \Xi(C_{Q'})|^p \delta_Q^{2-p} \lesssim \sum_{ C \in \cC_0} | \Xi(C) - \Xi(\pi(C))|^p \sum_{(Q, Q') \in \cQ_C} \delta_Q^{2-p},
\]
where $\cQ_C$ is defined in Lemma \ref{lem: cluster}. Applying Lemma \ref{lem: cluster} establishes that $\| G \|_{L^{2,p}(Q^0)} \lesssim_p \| \Xi \|_{L^{1,p}(\cC)}$. We then extend $G$ to a function $\widetilde{G}$ on all of $\R^2$ satisfying $\widetilde{G}|_{Q^0} = G|_{Q^0}$ and $\| \widetilde{G} \|_{L^{2,p}(\R^2)} \lesssim_p \| G \|_{L^{2,p}(Q^0)}$. Then $G|_E = f$, and we have established \eqref{eq: tree 3}, completing the proof of \ref{phi3}.

\bibliography{ref}
\bibliographystyle{plain}

\end{document}